%% file: main.tex
\documentclass[12pt]{elsarticle}
\usepackage[a4paper,top=3cm,bottom=3cm,left=3cm,right=3cm]{geometry}
\usepackage[utf8]{inputenc}

\input{0_preamble/0_packages}
\input{0_preamble/0_environment_style}
\input{0_preamble/0_new_commands}

\makeatletter
\def\ps@pprintTitle{ 
 \let\@oddhead\@empty
 \let\@evenhead\@empty
 \def\@oddfoot{\centerline{\thepage}}
 \let\@evenfoot\@oddfoot}
\makeatother

\usepackage{fancyhdr}
\usepackage{setspace}

\begin{document}

\begin{frontmatter}

\title{\textbf{Left-right symmetry of finite finitistic dimension}}

\author{Charley Cummings}

\begin{abstract}
We show that the finitistic dimension conjecture holds for all finite dimensional algebras if and only if, for all finite dimensional algebras, the finitistic dimension of an algebra being finite implies that the finitistic dimension of its opposite algebra is also finite. We also prove the equivalent statement for injective generation.
\end{abstract}


\end{frontmatter}

\setcounter{tocdepth}{1}
\tableofcontents

\input{content_files/intro}

\input{content_files/notation}

\input{content_files/construction}

\input{content_files/fin_dim}

\input{content_files/IG}

\input{bibliography}

\bigskip

\noindent 
\footnotesize \textsc{C. Cummings, Department of Mathematics, Aarhus University, Denmark.}
\\
\noindent  \textit{Email address:} 
{\texttt{c.cummings@math.au.dk}}

\end{document}

%% file: 0_preamble/0_packages.tex
\usepackage[dvipsnames]{xcolor}

\usepackage{amsmath}

\usepackage{amssymb}

\usepackage{amsthm}

\usepackage{arydshln}

\usepackage{colortbl}

\usepackage{fancyhdr}

\usepackage{enumerate} 

\usepackage{enumitem}

\usepackage[hidelinks]{hyperref}

\usepackage{tikz}

\usepackage{setspace}

\setlist{parsep=0pt,listparindent=\parindent}

%% file: 0_preamble/0_environment_style.tex
\newtheorem{theorem}{Theorem}[section]
\newtheorem{alphtheorem}{Theorem}

\newtheorem{proposition}[theorem]{Proposition}

\newtheorem{lemma}[theorem]{Lemma}
\theoremstyle{definition}
\newtheorem{definition}[theorem]{Definition}
\newtheorem{construction}[theorem]{Construction}
\newtheorem{example}[theorem]{Example}
\newtheorem{remark}[theorem]{Remark}

%% file: 0_preamble/0_new_commands.tex
\newcommand{\Mod}[1]{{\text{Mod-}{#1}}}
\newcommand{\fgmod}[1]{{\text{mod-}{#1}}}

\newcommand{\Inj}[1]{\text{Inj-}{#1}}

\newcommand{\der}[1]{\mathcal{D}({#1})}

\DeclareMathOperator{\Hom}{Hom}
\newcommand{\tens}[3]{{#1} {\otimes_{#2}} {#3}}

\DeclareMathOperator{\OpSpan}{span}
\newcommand{\vspan}[2]{\OpSpan_{#1}\{{#2}\}}

\DeclareMathOperator{\OpLoc}{Loc}
\newcommand{\lsub}[2]{\OpLoc_{#1}({#2})}
\newcommand{\lsubinj}[1]{\lsub{#1}{\Inj{#1}}}

\newcommand{\ds}[3]{\bigoplus_{{#1} \in {#2}}{#3}}

\newcommand{\dual}[1]{D(#1)}

\newcommand{\twobytwo}[4]{\left( \begin{matrix} {#1} & {#2} \\ {#3} & {#4} \end{matrix} \right)}
\newcommand{\onebytwo}[2]{\left( \begin{matrix} {#1} & {#2} \end{matrix} \right)}
\newcommand{\twobyone}[2]{\left( \begin{matrix} {#1} \\ {#2} \end{matrix} \right)}

\newcommand{\integers}[0]{\mathbb{Z}}

\newcommand{\quotient}[2]{{#1}/{#2}}

\newcommand{\function}[3]{{#1} \colon {#2} \rightarrow {#3}}

\DeclareMathOperator{\OpCohomology}{H}
\newcommand{\cohomology}[2]{\OpCohomology^{#1} \! \left( {#2} \right) }

\DeclareMathOperator{\OpRadical}{rad}
\newcommand{\rad}[1]{\OpRadical ({#1})}
\newcommand{\radpower}[2]{\OpRadical^{#1} ({#2})}

\DeclareMathOperator{\OpProjDim}{pd}
\newcommand{\projdim}[2]{\OpProjDim_{#1} \left({#2} \right)}

\DeclareMathOperator{\OpFinDim}{FinDim}
\newcommand{\FinDim}[1]{\OpFinDim ({#1})}

\DeclareMathOperator{\Opfindim}{findim}
\newcommand{\findim}[1]{\Opfindim ({#1})}

\DeclareMathOperator{\Opop}{op}
\newcommand{\op}[1]{{#1}^{\Opop}}

\newcommand{\iidemp}[1]{e_{#1}}
\newcommand{\viidemp}[1]{\tilde{e}_{#1}}

\newcommand{\constr}[0]{\tilde{A}}
\newcommand{\constrquiver}[0]{\tilde{Q}}
\newcommand{\constrrel}[0]{\tilde{I}}

\newcommand{\opconstr}[0]{\op{\constr}}

\newcommand{\loopedge}[1]{\tilde{\ell}_{#1}}
\newcommand{\connedge}[1]{\tilde{m}_{#1}}

\newcommand{\simp}[2][]{S_{#1}({#2})}
\newcommand{\alginj}[2]{\dual{{#1}{#2}}}

%% file: content_files/intro.tex
\section{Introduction} \label{sec:intro}

The homological conjectures are a collection of longstanding open questions about the representation theory of finite dimensional algebras. A summary can be found in \cite{Happel1991}. One of the strongest conjectures is the finitistic dimension conjecture, which concerns the finiteness of a numerical invariant associated to an algebra. If it is true, then several other homological conjectures also hold, including the Nunke condition \cite{Nunke--1955--NunkeCondition} and the Nakayama conjecture \cite{Nakayama--1958--NakConj--WithMullerResults}.

Numerical invariants have long been used to measure the complexity of mathematical structures. Some examples in representation theory, are homological dimensions such as the global dimension \cite{Cartan1956}, the finitistic dimension \cite{AuslanderBuchsbaum1957}, and the dominant dimension \cite{Nakayama--1958--NakConj--WithMullerResults}. These three dimensions have been studied extensively, particularly with a view to classify when they are finite. To have finite global dimension is quite a restrictive property for a ring; even straightforward algebras like $\quotient{k[\ell]}{\langle \ell^2 \rangle}$ do not satisfy this property. Whereas, the other two dimensions are conjectured to be finite for quite large classes of rings. Indeed, the Nakayama conjecture asserts that the dominant dimension of a finite dimensional algebra is infinite if and only if the algebra is self-injective \cite{Nakayama--1958--NakConj--WithMullerResults}. Whereas, the finitistic dimension conjecture asserts that the finitistic dimension of a finite dimensional algebra is finite. A survey on this conjecture can be found in \cite{Huisgen1995}. These conjectures remain open in general, but have been verified for various classes of algebras including commutative algebras \cite{AuslanderBuchsbaum1957}, radical cube zero algebras \cite{GreenZimmermannHuisgen1991}, monomial algebras \cite{GreenKirkmanKuzmanovich1991-MonomialAlgebras-LeftRightSymmetry}, and representation dimension three algebras \cite{IgusaTodorov2005}.

Homological dimensions can be defined in terms of either left or right modules. Sometimes, the value of a dimension is unaffected by this choice, that is, the dimension is left-right symmetric. This is true for both the global dimension of a Noetherian ring \cite[Theorem~4]{Auslander--1955--LRGlobDimEqual}, and the dominant dimension of a finite dimensional algebra \cite[Theorem~4]{Muller1968-DominantDimensionIsLeftRightSymmetric}. However, it is not true for the finitistic dimension; not even when restricted to quiver algebras \cite[Example~2.2]{Jensen--Lenzing--1982--LRFinDimDifferentExample}. In fact, the left and right finitistic dimensions of a quiver algebra can be arbitrarily different  \cite[Example~1.2]{GreenKirkmanKuzmanovich1991-MonomialAlgebras-LeftRightSymmetry}.

For a ring $\Lambda$, a left $\Lambda$-module is simply a right $\op{\Lambda}$-module. Since the dominant dimension of a finite dimensional algebra is left-right symmetric, it follows that an algebra satisfies the Nakayama conjecture if and only if its opposite algebra does. In 1991, Happel \cite{Happel1991} asked if this property also holds for the finitistic dimension conjecture. That is, if an algebra has finite finitistic dimension, then does its opposite algebra also have finite finitistic dimension? The question has remained unsolved for the last three decades \cite[p.27]{Christensen--Holm--2010--LRfindim}. 
We show that, heuristically, this is for good reason; answering Happel's question is equivalent to proving or disproving the finitistic dimension conjecture.

\begin{alphtheorem} [Theorem~\ref{thm:little_findim}] \label{thm:intro_little_findim}
    The finitistic dimension conjecture holds for all finite dimensional algebras if and only if, for all finite dimensional algebras $\Lambda$, the finitistic dimension of $\Lambda$ being finite implies that the finitistic dimension of $\op{\Lambda}$ is finite.
\end{alphtheorem}

To prove Theorem~\ref{thm:intro_little_findim}, we define a ring construction that takes a basic finite dimensional algebra $A$ and produces a related algebra $\constr$. The finitistic dimension of $\constr$ bounds the finitistic dimension of $A$ from above. Moreover, the algebra $\constr$ is constructed in such a way that its opposite algebra has quite simple homological properties. This allows us to prove something stronger than Theorem~\ref{thm:intro_little_findim}.

\begin{alphtheorem} [Propositions~\ref{prop:findim(A)_less_than_findim(tildeA)} and \ref{prop:findim(tildeA^op)_is_zero}] \label{thm:intro_contapositive_findim}
    If there exists a counterexample to the finitistic dimension conjecture, then there exists a finite dimensional algebra $\Lambda$ such that the finitistic dimension of $\Lambda$ is infinite and the finitistic dimension of $\op{\Lambda}$ is zero.
\end{alphtheorem}

A natural setting for many homological properties of rings and their modules is the derived category. Indeed, a Noetherian ring has finite global dimension if and only if its bounded derived category is generated, as a triangulated category, by the injective modules. Recently, Rickard proved that, for a finite dimensional algebra, finite finitistic dimension can also be seen as a property of the derived category \cite[Theorem~4.4]{Rickard2019}. More generally, if the unbounded derived category of a finite dimensional algebra is generated, as a triangulated category with set indexed coproducts, by the injective modules, then the finitistic dimension conjecture holds for that algebra \cite[Theorem~4.3]{Rickard2019}. In Section~\ref{sec:IG}, we use the construction $\constr$ to prove an equivalent statement to Theorem~\ref{thm:intro_little_findim} for this generation property.

\subsection*{Outline}

In Section~\ref{sec:constr}, we take a basic finite dimensional algebra $A$ over an algebraically closed field and define a related algebra $\constr$ whose homological properties are connected to those of $A$. In Section~\ref{sec:findim}, we focus on the finitistic dimensions of $A$, $\constr$, and $\opconstr$ and prove Theorem~\ref{thm:intro_little_findim}. Section~\ref{sec:IG} has the same structure as Section~\ref{sec:findim} except we focus on injective generation rather than the finitistic dimension.

\subsection*{Acknowledgements}

I would like to thank my PhD supervisor Jeremy Rickard for several useful discussions regarding this paper. I would also like to thank Carlo Klapproth and Henning Krause for conversations and email exchanges related to the contents of this paper. In particular, I am grateful to Henning Krause for communicating a simplification in the construction of $\constr$ that simplifies subsequent notation and proofs. Finally, I would like to thank Tim Burness and Peter J{\o}rgensen for their comments on preliminary versions of this paper.

%% file: content_files/notation.tex
\subsection*{Notation}

All rings and ring homomorphisms are unital, and modules are right modules unless otherwise stated.  Let $\Lambda$ be a finite dimensional algebra over a field $k$ with $M_{\Lambda}$ a right $\Lambda$-module.

\begin{itemize}
	\item $\Mod{{\Lambda}}$ ($\fgmod{{\Lambda}}$) is the category of (finitely generated) right ${\Lambda}$-modules.
	\item $\projdim{{\Lambda}}{M}$ is the projective dimension of $M$.
	\item Let $e \in \Lambda$ be an idempotent.
	\begin{itemize}
		\item $\dual{{\Lambda}e}$ is the injective right ${\Lambda}$-module $\Hom_k(\Hom_{\Lambda}(e\Lambda,\Lambda),k)$.
		\item $\simp[\Lambda]{e}$ is the semisimple right $\Lambda$-module $\quotient{e{\Lambda}}{\rad{e{\Lambda}}}$. When the ring $\Lambda$ is clear from context we omit the subscript. 
	\end{itemize}
	\item $\der{{\Lambda}}$ is the unbounded derived category of cochain complexes of right ${\Lambda}$-modules.
	\item $\Inj{{\Lambda}}$ is the category of injective right ${\Lambda}$-modules.
\end{itemize}

%% file: content_files/construction.tex
\section{The construction of \texorpdfstring{$\constr$}{A}} \label{sec:constr}

In this section, we take a finite dimensional algebra $A$ and construct a new, related algebra $\constr$. Since the finitistic dimension is invariant under field extensions \cite[Theorem~2.5]{Jensen--Lenzing--1982--LRFinDimDifferentExample} and Morita equivalence, we define $\constr$ for a basic finite dimensional algebra $A$ over an algebraically closed field $k$. To each simple $A$-module, we adjoin a copy of the finite dimensional algebra $\quotient{k[\loopedge{}]}{\langle\loopedge{}^2\rangle}$. We do this in such a way that $\constr$ can be realised as a triangular matrix ring where $A$ is one of the diagonal entries.

\begin{construction} \label{constr:tildeA}
	Let $A$ be a basic finite dimensional algebra over an algebraically closed field $k$. Let $\{e_i : 1 \leq i \leq n\} \subset A$ be a set of primitive idempotents such that $\{\simp{e_i} : 1 \leq i \leq n\}$ is an irredundant set of isomorphism classes of the simple right $A$-modules. Let $M$ be the $n$-dimensional free $k$-module with basis $\{\connedge{i} : 1 \leq i \leq n\}$. Define a right $k$-linear $A$-module structure on $M$ by
	\begin{equation*}
		\connedge{i}\iidemp{i} = \connedge{i} \text{ and } \connedge{i}a=0 \text{ for } a \in \rad{A}.
	\end{equation*}
	For each $1 \leq i \leq n$, let $B_i$ denote the local ring $\quotient{k[\loopedge{i}]}{\langle \loopedge{i}^2\rangle}$ and let $\viidemp{i}$ denote its multiplicative identity. Let $B$ denote the product ring $\prod_{i=1}^n B_i$.
	Define a left $k$-linear $B$-module structure on $M$ by
	\begin{equation*}
		\viidemp{i}\connedge{i} = \connedge{i} \text{ and } b\connedge{i}=0 \text{ for } b \in \rad{B}.
	\end{equation*}
	Then $\constr$ is the triangular matrix ring
	 \begin{equation*}
		\twobytwo{A}{0}{_BM_A}{B}.
	\end{equation*}
\end{construction}

When $A$ is a quiver algebra, Construction~\ref{constr:tildeA} can be realised via simple graph operations on the underlying quiver.

\begin{example} \label{eg:construction_of_tildeA}
    Let $A$ be the quiver algebra $\quotient{kQ}{I}$ where $k$ is an algebraically closed field, $Q$ is the quiver
    \begin{center}
        \begin{tikzpicture}
            \node (1) at (0,0) {$1$};
            \node (2) at (1,-2) {$2$};
            \node (3) at (-1,-2) {$3$};

            \draw[<-] (1) to [bend right = 20] node[midway,left] {$\delta$} (3);
            \draw[<-] (1) to [bend right = 20] node[midway,left] {$\gamma$} (2);
            \draw[<-] (2) to [bend right = 20] node[midway,right] {$\beta$} (1);
            \draw[<-] (2) to [bend left = 10] node[midway,above] {$\varepsilon$} (3);
            \draw[<-] (1) to [out = 45, in = 135, looseness = 5] node[midway,above] {$\alpha$} (1);
            \draw[<-] (3) to [out = -135, in = 135, looseness = 5] node[midway,left] {$\zeta$} (3);
        \end{tikzpicture}
    \end{center}
    and $I = \langle \alpha\beta, \beta\gamma, \gamma\alpha, \delta\beta,
    \zeta^2, \zeta\varepsilon,\alpha^3,  \varepsilon\gamma\beta, \delta\alpha^2,  \zeta\delta-\delta\alpha\rangle$. Then the indecomposable projective $A$-modules can be represented by the diagrams below. In these diagrams, the vertices represent the composition factors of the module and the arrows represent the action of the algebra on these composition factors.
    \begin{center}
        \begin{tikzpicture}
            \node (1) at (0,1) {$\iidemp{1}A$};
            \node (11) at (0,-0) {$1$};
            \node (12) at (-0.5,-1) {$1$};
            \node (13) at (0.5,-1) {$2$};
            \node (14) at (-0.5,-2) {$1$};

            \draw[->] (11) to node[midway,left] {$\alpha$} (12);
            \draw[->] (11) to node[midway,right] {$\beta$} (13);
            \draw[->] (12) to node[midway,left] {$\alpha$} (14);

            \begin{scope}[xshift = 1cm]
                \node (2) at (1.5,1) {$\iidemp{2}A$};

                \node (21) at (1.5,-0) {$2$};
                \node (22) at (1.5,-1) {$1$};
                \node (24) at (1.5,-2) {$2$};

                \draw[->] (21) to node[midway,right] {$\gamma$} (22);
                \draw[->] (22) to node[midway,right] {$\beta$} (24);
            \end{scope}

            \begin{scope}[xshift = 1cm]
                \node (3) at (4.5,1) {$\iidemp{3}{A}$};

                \node (31) at (4.5,-0) {$3$};
                \node (32) at (3.5,-1) {$2$};
                \node (33) at (3.5,-2) {$1$};
                \node (34) at (4.5,-1) {$1$};
                \node (35) at (5.5,-1) {$3$};
                \node (36) at (5,-2) {$1$};

                \draw[->] (31) to node[midway,left] {$\varepsilon$} (32);
                \draw[->] (32) to node[midway,left] {$\gamma$} (33);
                \draw[->] (31) to node[midway,left] {$\delta$} (34);
                \draw[->] (31) to node[midway,right] {$\zeta$} (35);
                \draw[->] (34) to node[midway,left] {$\alpha$} (36);
                \draw[->] (35) to node[midway,right] {$\delta$} (36);
            \end{scope}
        \end{tikzpicture}
    \end{center}
    
    The injective $A$-modules can be represented by the diagrams below.
    \begin{center}
        \begin{tikzpicture}
            \node (1) at (0,0) {$\alginj{A}{e_1}$};

            \node (11) at (0,1) {$1$};
            \node (12) at (-1,2) {$2$};
            \node (13) at (0,2) {$1$};
            \node (14) at (1,2) {$3$};
            \node (15) at (-0.5,3) {$1$};
            \node (17) at (0.5,3) {$3$};
            \node (19) at (-1.5,3) {$3$};

            \draw[<-] (11) to node[midway,left] {$\gamma$} (12);
            \draw[<-] (11) to node[midway,left] {$\alpha$} (13);
            \draw[<-] (11) to node[midway,right] {$\delta$} (14);
            \draw[<-] (13) to node[midway,left] {$\alpha$} (15);
            \draw[<-] (13) to node[midway,right] {$\delta$} (17);
            \draw[<-] (14) to node[midway,right] {$\zeta$} (17);
            \draw[<-] (12) to node[midway,left] {$\varepsilon$} (19);

            \node (2) at (3.5,0) {$\alginj{A}{e_2}$};

            \node (21) at (3.5,1) {$2$};
            \node (22) at (3,2) {$1$};
            \node (23) at (3,3) {$2$};
            \node (24) at (4,2) {$3$};

            \draw[<-] (21) to node[midway,left] {$\beta$} (22);
            \draw[<-] (21) to node[midway,right] {$\varepsilon$} (24);
            \draw[<-] (22) to node[midway,left] {$\gamma$} (23);

            \node (3) at (6,0) {$\alginj{A}{e_3}$};

            \node (31) at (6,1) {$3$};
            \node (32) at (6,2) {$3$};

            \draw[<-] (31) to node[midway,right] {$\zeta$} (32);
        \end{tikzpicture}
    \end{center}

    The simple $A$-modules correspond to the vertices of the quiver $Q$. Therefore, in this case, $M$ is a $3$-dimensional $k$-vector space with basis $\{\connedge{1}, \connedge{2}, \connedge{3} \} $, where $\viidemp{i} \connedge{i} \iidemp{i} = \connedge{i}$ for each $1 \leq i \leq 3$. Moreover, $B$ is a product of three copies of $\quotient{k[\loopedge{}]}{\langle\loopedge{}^2\rangle}$ and $\rad{\quotient{k[\loopedge{}]}{\langle\loopedge{}^2\rangle}}$ is equal to $\vspan{k}{\loopedge{}}$. So, we can realise $\constr$ as the quiver algebra  $\quotient{k\constrquiver}{\constrrel}$ where $\constrquiver$ is the quiver
    \begin{center}
        \begin{tikzpicture}
            \node (1) at (0,0) {$1$};
            \node (2) at (1,-2) {$2$};
            \node (3) at (-1,-2) {$3$};

            \draw[<-] (1) to [bend right = 20] node[midway,left] {$\delta$} (3);
            \draw[<-] (1) to [bend right = 20] node[midway,left] {$\gamma$} (2);
            \draw[<-] (2) to [bend right = 20] node[midway,right] {$\beta$} (1);
            \draw[<-] (2) to [bend left = 10] node[midway,above] {$\varepsilon$} (3);
            \draw[<-] (1) to [out = 45, in = 135, looseness = 5] node[midway,above] {$\alpha$} (1);
            \draw[<-] (3) to [out = -135, in = 135, looseness = 5] node[midway,left] {$\zeta$} (3);

            \node (v1) at (1.5,0) {\textcolor{Emerald}{$\tilde{1}$}};

            \node (v2) at (2.5,-2) {\textcolor{Emerald}{$\tilde{2}$}};

            \node (v3) at (-1,-3.5) {\textcolor{Emerald}{$\tilde{3}$}};

            \draw[<-,Emerald] (1) to node[midway,above] {$\connedge{1}$} (v1);
            \draw[->,Emerald] (v1) to [out = -45, in = 45, looseness = 5] node[midway,right] {$\loopedge{1}$} (v1);

            \draw[<-,Emerald] (2) to node[midway,above] {$\connedge{2}$} (v2);
            \draw[->,Emerald] (v2) to [out = -45, in = 45, looseness = 5] node[midway,right] {$\loopedge{2}$} (v2);

            \draw[<-,Emerald] (3) to node[midway,right] {$\connedge{3}$} (v3);
            \draw[->,Emerald] (v3) to [out = -135, in = -45, looseness = 5] node[midway,below] {$\loopedge{3}$} (v3);
        \end{tikzpicture}
    \end{center}
    and
    \begin{equation*}
        \constrrel = \langle r, \connedge{i}a, \loopedge{i}\connedge{i}, \loopedge{i}^2 : r \in I, a \in \rad{A}, 1 \leq i \leq 3 \rangle.
    \end{equation*}
    The projective $\constr$-modules can be represented by the diagrams below.
    \begin{center}
        \begin{tikzpicture}
            \node (1) at (0,1) {$\iidemp{1}\constr$};
            \node (11) at (0,0) {$1$};
            \node (12) at (-0.5,-1) {$1$};
            \node (13) at (0.5,-1) {$2$};
            \node (14) at (-0.5,-2) {$1$};

            \draw[->] (11) to (12);
            \draw[->] (11) to (13);
            \draw[->] (12) to (14);
            
            \begin{scope}[xshift = 1cm]
                \node (2) at (1,1) {$\iidemp{2}\constr$};

                \node (21) at (1,-0) {$2$};
                \node (22) at (1,-1) {$1$};
                \node (24) at (1,-2) {$2$};

                \draw[->] (21) to (22);
                \draw[->] (22) to (24);
            \end{scope}

            \begin{scope}[xshift = 0cm]
                \node (3) at (4.5,1) {$\iidemp{3}\constr$};

                \node (31) at (4.5,-0) {$3$};
                \node (32) at (3.5,-1) {$2$};
                \node (33) at (3.5,-2) {$1$};
                \node (34) at (4.5,-1) {$1$};
                \node (35) at (5.5,-1) {$3$};
                \node (36) at (5,-2) {$1$};

                \draw[->] (31) to (32);
                \draw[->] (32) to (33);
                \draw[->] (31) to (34);
                \draw[->] (31) to (35);
                \draw[->] (34) to (36);
                \draw[->] (35) to (36);
            \end{scope}

            \begin{scope}[xshift = 7.5cm]
                \node (v) at (0,1) {$\viidemp{1}\constr$};

                \node (v1) at (0,0) {$\textcolor{Emerald}{\tilde{1}}$};
                \node (1v1) at (-0.5,-1) {$1$};
                \node (uv1) at (0.5,-1) {$\textcolor{Emerald}{\tilde{1}}$};

                \draw[->,Emerald] (v1) to node[midway,left] {\textcolor{Emerald}{$\connedge{1}$}} (1v1);
                \draw[->,Emerald] (v1) to node[midway,right] {\textcolor{Emerald}{$\loopedge{1}$}} (uv1);
            \end{scope}

            \begin{scope}[xshift = 10cm]
                \node (v) at (0,1) {$\viidemp{2}\constr$};

                \node (v1) at (0,0) {$\textcolor{Emerald}{\tilde{2}}$};
                \node (1v1) at (-0.5,-1) {$2$};
                \node (uv1) at (0.5,-1) {$\textcolor{Emerald}{\tilde{2}}$};

                \draw[->,Emerald] (v1) to node[midway,left] {\textcolor{Emerald}{$\connedge{2}$}} (1v1);
                \draw[->,Emerald] (v1) to node[midway,right] {\textcolor{Emerald}{$\loopedge{2}$}} (uv1);
            \end{scope}

            \begin{scope}[xshift = 12.5cm]
                \node (v) at (0,1) {$\viidemp{3}\constr$};

                \node (v1) at (0,0) {$\textcolor{Emerald}{\tilde{3}}$};
                \node (1v1) at (-0.5,-1) {$3$};
                \node (uv1) at (0.5,-1) {$\textcolor{Emerald}{\tilde{3}}$};

                \draw[->,Emerald] (v1) to node[midway,left] {\textcolor{Emerald}{$\connedge{3}$}} (1v1);
                \draw[->,Emerald] (v1) to node[midway,right] {\textcolor{Emerald}{$\loopedge{3}$}} (uv1);
            \end{scope}
        \end{tikzpicture}
    \end{center}
    The injective $\constr$-modules can be represented by the diagrams below.
    \begin{center}
        \begin{tikzpicture}
            \node (1) at (0,0) {$\alginj{\constr}{e_1}$};

            \node (11) at (0,1) {$1$};
            \node (12) at (-0.5,2) {$2$};
            \node (13) at (0.5,2) {$1$};
            \node (14) at (1.5,2) {$3$};
            \node (15) at (0,3) {$1$};
            \node (17) at (1,3) {$3$};
            \node (19) at (-1,3) {$3$};

            \draw[<-] (11) to (12);
            \draw[<-] (11) to (13);
            \draw[<-] (11) to (14);
            \draw[<-] (13) to (15);
            \draw[<-] (13) to (17);
            \draw[<-] (14) to (17);
            \draw[<-] (12) to (19);

            \node (110) at (-1.5,2) {$\textcolor{Emerald}{\tilde{1}}$};

            \draw[<-,Emerald] (11) to node[midway,left] {$\textcolor{Emerald}{\connedge{1}}$} (110);

            \node (2) at (4,0) {$\alginj{\constr}{e_2}$};

            \node (21) at (4,1) {$2$};
            \node (22) at (4,2) {$1$};
            \node (23) at (4,3) {$2$};
            \node (24) at (5,2) {$3$};

            \draw[<-] (21) to (22);
            \draw[<-] (21) to (24);
            \draw[<-] (22) to (23);

            \node (27) at (3,2) {$\textcolor{Emerald}{\tilde{2}}$};

            \draw[<-,Emerald] (21) to node[midway,left] {$\textcolor{Emerald}{\connedge{2}}$} (27);

            \node (3) at (7,0) {$\alginj{\constr}{e_3}$};

            \node (31) at (7,1) {$3$};
            \node (32) at (7.5,2) {$3$};

            \draw[<-] (31) to (32);

            \node (33) at (6.5,2) {$\textcolor{Emerald}{\tilde{3}}$};

            \draw[<-,Emerald] (31) to node[midway,left] {$\textcolor{Emerald}{\connedge{3}}$} (33);

            \begin{scope}[xshift = 9cm]
                \node (v1) at (0,0) {$\alginj{\constr}{\viidemp{1}}$};

                \node (v11) at (0,1) {$\textcolor{Emerald}{\tilde{1}}$};
                \node (v13) at (0,2) {$\textcolor{Emerald}{\tilde{1}}$};

                \draw[<-,Emerald] (v11) to node[midway,right] {$\textcolor{Emerald}{\loopedge{1}}$} (v13);

                \node (v2) at (1.5,0) {$\alginj{\constr}{\viidemp{2}}$};

                \node (v21) at (1.5,1) {$\textcolor{Emerald}{\tilde{2}}$};
                \node (v23) at (1.5,2) {$\textcolor{Emerald}{\tilde{2}}$};

                \draw[<-,Emerald] (v21) to node[midway,right] {$\textcolor{Emerald}{\loopedge{2}}$} (v23);

                \node (v3) at (3,0) {$\alginj{\constr}{\viidemp{3}}$};

                \node (v31) at (3,1) {$\textcolor{Emerald}{\tilde{3}}$};
                \node (v33) at (3,2) {$\textcolor{Emerald}{\tilde{3}}$};

                \draw[<-,Emerald] (v31) to node[midway,right] {$\textcolor{Emerald}{\loopedge{3}}$} (v33);
            \end{scope}
        \end{tikzpicture}
    \end{center}
\end{example}

\begin{lemma} \label{lemma:simples_in_top_of_dual}
    Let $A$ be a basic finite dimensional algebra over an algebraically closed field $k$, with $\constr$ defined as in Construction~\ref{constr:tildeA}. Then each simple right $\opconstr$-module is isomorphic to $\simp{f}$ for some $f \in \{\iidemp{i}, \viidemp{i} : 1 \leq i \leq n \}$. Moreover, for each $1 \leq i \leq n$,
    there exists a short exact sequence of right $\opconstr$-modules
    \begin{equation*}
        0 \xrightarrow[\hspace{3.2em}]{} 
        \simp{\viidemp{i}} \xrightarrow[\hspace{3.2em}]{} 
        \alginj{\opconstr}{\viidemp{i}} \xrightarrow[\hspace{3.25em}]{\onebytwo{\connedge{i}}{\loopedge{i}}}
        \simp{\iidemp{i}} \oplus \simp{\viidemp{i}}
        \xrightarrow[\hspace{3.2em}]{} 
        0.
    \end{equation*}
\end{lemma}

\begin{proof}
    The primitive idempotents of a triangular matrix ring correspond to the primitive idempotents of its diagonal components. Therefore, the primitive idempotents of $\opconstr$ are $\iidemp{i}$ and $\viidemp{i}$ for $1 \leq i \leq n$. Consequently, the simple right $\opconstr$-modules are of the form $\simp[\opconstr]{f}$ where $f$ is either $\iidemp{i}$ or $\viidemp{i}$ for some $1 \leq i \leq n$.

    For each $1 \leq i \leq n$, the radical of $B_i$ is $\vspan{k}{\loopedge{i}}$. So, it follows that the radical of $\constr$ is 
    \begin{equation*}
        \vspan{k}{ a, \connedge{i}, \loopedge{i} : a \in \rad{A}, 1 \leq i \leq n}.
    \end{equation*}
    Fix $1 \leq i \leq n$. Then $\rad{\viidemp{i}\constr}$ is equal to $\vspan{k}{\connedge{i}, \loopedge{i}}$ as an $\constr$-module. Moreover,  $\radpower{2}{\viidemp{i}\constr}$ is trivial, so $\rad{\viidemp{i}\constr}$ is semisimple. Since the primitive idempotents of $\constr$ are $\iidemp{j}$ and $\viidemp{j}$ for $1 \leq j \leq n$ and
    \begin{equation*}
        \connedge{i} \iidemp{i} = \connedge{i} \text{ and } \loopedge{i} \viidemp{i} = \loopedge{i},
    \end{equation*}
    it follows that $\rad{\viidemp{i}\constr}$ decomposes as a right $\constr$-module to $\simp[\constr]{\iidemp{i}} \oplus \simp[\constr]{\viidemp{i}}$. So, there exists a short exact sequence of right $\constr$-modules
    \begin{equation*}
        0 \xrightarrow[\hspace{3.2em}]{} 
        \simp[\constr]{\iidemp{i}} \oplus \simp[\constr]{\viidemp{i}} \xrightarrow[\hspace{3.2em}]{\twobyone{\connedge{i}}{\loopedge{i}}} 
        \viidemp{i}\constr \xrightarrow[\hspace{3.25em}]{}
        \simp[\constr]{\viidemp{i}}
        \xrightarrow[\hspace{3.2em}]{} 
        0.
    \end{equation*}
    Note that $\viidemp{i}\constr$ is equal to $\opconstr\viidemp{i}$ and $\constr$ is a finite dimensional algebra over $k$. So, the application of the duality functor 
    \begin{equation*}
        D = \function{\Hom_k(-,k)}{\fgmod{\constr}}{\fgmod{\opconstr}}
    \end{equation*}
    to this sequence yields the required short exact sequence of right $\opconstr$-modules.
\end{proof}

%% file: content_files/fin_dim.tex
\section{The finitistic dimension} \label{sec:findim}

In this section, we focus on the properties of the finitistic dimensions of the finite dimensional algebras $A$, $\constr$, and $\opconstr$.  We start with the definition of the finitistic dimension.

\begin{definition}[The (little/big) finitistic dimension] \label{defn:findim}
    Let $\Lambda$ be a finite dimensional algebra. Then the little finitistic dimension of $\Lambda$ is
    \begin{equation*}
        \findim{\Lambda} = \sup \{ \projdim{\Lambda}{M} : M \in \fgmod{\Lambda} \text{ and } \projdim{\Lambda}{M} < \infty  \}.
    \end{equation*}
    The big finitistic dimension of $\Lambda$ is
    \begin{equation*}
        \FinDim{\Lambda} = \sup \{ \projdim{\Lambda}{M} : M \in \Mod{\Lambda} \text{ and } \projdim{\Lambda}{M} < \infty \}.
    \end{equation*}
\end{definition}

The main result of this section, namely Theorem~\ref{thm:little_findim}, concerns the relationship between the finitistic dimension of a finite dimensional algebra and its opposite algebra. The proof of the result is the combination of two propositions that we state and prove now. In particular, we show that  that $\findim{A} \leq \findim{\constr}$ in Proposition~\ref{prop:findim(A)_less_than_findim(tildeA)} and that $\findim{\opconstr} = 0$ in Proposition~\ref{prop:findim(tildeA^op)_is_zero}.

\begin{proposition} \label{prop:findim(A)_less_than_findim(tildeA)}
    Let $A$ be a basic finite dimensional algebra over an algebraically closed field, with $\constr$ defined as in Construction~\ref{constr:tildeA}. Then $\findim{A} \leq \findim{\constr}$ and $\FinDim{A} \leq \FinDim{\constr}$.
\end{proposition}

\begin{proof}
    Since $\constr$ is a triangular matrix ring, the statements follow by \cite[Corollary~4.21]{FossumGriffithReiten1975}.
\end{proof}

\begin{proposition} \label{prop:findim(tildeA^op)_is_zero}
    Let $A$ be a basic finite dimensional algebra over an algebraically closed field, with $\constr$ defined as in Construction~\ref{constr:tildeA}. Then $\findim{\opconstr} = 0$ and $\FinDim{\opconstr} = 0$.
\end{proposition}

\begin{proof}
	Since $\findim{\opconstr} \leq \FinDim{\opconstr}$, it suffices to prove that $\FinDim{\opconstr} = 0$. Suppose that $\Lambda$ is a finite dimensional algebra over a field. Then $\FinDim{\Lambda}=0$ if and only if there is a non-zero homomorphism from the dual module $\dual{\Lambda}$ to every simple $\Lambda$-module \cite[Lemma~6.2]{Bass1960}. Therefore, by Lemma~\ref{lemma:simples_in_top_of_dual}, it follows that $\FinDim{\opconstr}=0$.
\end{proof}

\begin{theorem} \label{thm:little_findim}
    The little finitistic dimension conjecture holds for all finite dimensional algebras if and only if, for all finite dimensional algebras $\Lambda$, $\findim{\Lambda}<\infty$ implies that $\findim{\op{\Lambda}} < \infty$.
\end{theorem}

\begin{proof}
    The forward direction holds trivially, so we focus on the reverse direction. 
    
    Let $C$ be a finite dimensional algebra over a field $k$. Then $\bar{C} = \tens{C}{k}{\bar{k}}$ is a finite dimensional algebra over the algebraically closed field $\bar{k}$, and $\findim{C} \leq \findim{\bar{C}}$ by \cite[Proposition~2.1]{Jensen--Lenzing--1982--LRFinDimDifferentExample}. Every finite dimensional algebra is Morita equivalent to a basic finite dimensional algebra and the finitistic dimension is invariant under Morita equivalence. So, there exists a basic finite dimensional algebra $A$ over the field $\bar{k}$ such that $\findim{C} \leq \findim{A}$.
    
    Let $\constr$ be defined as in Construction~\ref{constr:tildeA}. Then $\findim{\opconstr} =0$ by Proposition~\ref{prop:findim(tildeA^op)_is_zero}. Hence,
    \begin{equation*}
        \findim{\constr} = \findim{\op{(\opconstr)}} < \infty.
    \end{equation*}
    So, by Proposition~\ref{prop:findim(A)_less_than_findim(tildeA)}, $\findim{A} < \infty$, and the finitistic dimension conjecture holds for $C$.
\end{proof}

\begin{theorem} \label{thm:big_findim}
    The big finitistic dimension conjecture holds for all finite dimensional algebras if and only if, for all finite dimensional algebras $\Lambda$, $\FinDim{\Lambda}<\infty$ implies that $\FinDim{\op{\Lambda}} < \infty$.
\end{theorem}

\begin{proof}
    This is the same as the proof of Theorem~\ref{thm:little_findim}, using \cite[Proposition~2.1]{Jensen--Lenzing--1982--LRFinDimDifferentExample} and Propositions~\ref{prop:findim(A)_less_than_findim(tildeA)} and \ref{prop:findim(tildeA^op)_is_zero}.
\end{proof}

\begin{remark}
    If the little finitistic dimension conjecture fails, then there exists a finite dimensional algebra $A$ such that $\findim{\constr} = \infty$, but $\findim{\opconstr}=0$ by Propositions~\ref{prop:findim(A)_less_than_findim(tildeA)} and \ref{prop:findim(tildeA^op)_is_zero}. A similar argument holds for the big finitistic dimension.
\end{remark}

%% file: content_files/IG.tex
\section{Injective generation} \label{sec:IG}

This section has the same structure as Section~\ref{sec:findim} except we focus on generation properties of the derived categories of the finite dimensional algebras $A$, $\constr$, and $\opconstr$. We start with the definition of a localising subcategory.

\begin{definition}[Localising subcategory] \label{defn:localising_subcategory}
    Let $\Lambda$ be a finite dimensional algebra over a field. A triangulated subcategory of $\der{\Lambda}$ is a localising subcategory if it is closed under set indexed coproducts. Let $\mathcal{S} = \{X_i\}_{i \in I}$ be a set of complexes of $\Lambda$-modules. Denote the smallest localising subcategory that contains $\mathcal{S}$ by $\lsub{\der{\Lambda}}{\mathcal{S}}$.
\end{definition}

It is well-known that the derived category is compactly generated or, equivalently, that the smallest localising subcategory that contains the regular module is the entire derived category. We ask when this property holds for the injective modules.

\begin{definition}[Injectives generate] \label{defn:injectives_generate}
    Let $\Lambda$ be a finite dimensional algebra over a field. Then injectives generate for $\Lambda$ if $\lsubinj{\Lambda} = \der{\Lambda}$.
\end{definition}

This property is strongly connected to the homological conjectures. In particular, if injectives generate for a finite dimensional algebra over a field, then the big finitistic dimension conjecture holds for that algebra \cite[Theorem~4.3]{Rickard2019}. As such, it is perhaps unsurprising that an analogous set of results to those in Section~\ref{sec:findim} hold in the more general setting of injective generation.

\begin{proposition} \label{prop:IG(tildeA)_implies_IG(A)}
    Let $A$ be a basic finite dimensional algebra over an algebraically closed field, with $\constr$ defined as in Construction~\ref{constr:tildeA}. If injectives generate for $\tilde{A}$, then injectives generate for $A$.
\end{proposition}

\begin{proof}
    Since $\constr$ is a triangular matrix ring, the statement holds by \cite[Example~6.11]{Cummings2021}.
\end{proof}

\begin{proposition} \label{prop:IG(tildeA^op)}
    Let $A$ be a basic finite dimensional algebra over an algebraically closed field, with $\constr$ defined as in Construction~\ref{constr:tildeA}. Then injectives generate for $\opconstr$.
\end{proposition}

\begin{proof}
    Since $\opconstr$ is a finite dimensional algebra over a field, the smallest localising subcategory of $\der{\opconstr}$ that contains the simple $\opconstr$-modules is equal to $\der{\opconstr}$ by \cite[Lemma~6.1]{Rickard2019}. So, it suffices to show that the simple $\opconstr$-modules lie in $\lsubinj{\opconstr}$.
    
    Fix $1 \leq i \leq n$. By Lemma~\ref{lemma:simples_in_top_of_dual}, there exists a short exact sequence of $\opconstr$-modules
    \begin{equation*} \label{eq:cohomologyxi}
        0 \xrightarrow[\hspace{3.2em}]{} 
        \simp{\viidemp{i}} \xrightarrow[\hspace{3.2em}]{} 
        \alginj{\opconstr}{\viidemp{i}} \xrightarrow[\hspace{3.25em}]{\onebytwo{\connedge{i}}{\loopedge{i}}} 
        \simp{\iidemp{i}} \oplus \simp{\viidemp{i}}
        \xrightarrow[\hspace{3.2em}]{} 
        0.
    \end{equation*}
    Both the kernel and image of $(\connedge{i} \enspace \loopedge{i})$ contain a copy of $\simp{\viidemp{i}}$ as a direct summand. So, the short exact sequence can be spliced with itself to build a complex
    \begin{equation*}
        X_i = \cdots \rightarrow  \alginj{\opconstr}{\viidemp{i}} \xrightarrow{\loopedge{i}} \alginj{\opconstr}{\viidemp{i}} \xrightarrow{\loopedge{i}} \alginj{\opconstr}{\viidemp{i}} \rightarrow 0 \rightarrow 0 \rightarrow \cdots
    \end{equation*}
    whose cohomology modules are given by
    \begin{equation*}
        \cohomology{n}{X_i} \cong \begin{cases}
        \simp{\iidemp{i}} \oplus \simp{\viidemp{i}},& \text{if } n = 0,
        \\
        \simp{\iidemp{i}},& \text{if } n < 0,
        \\
        0,& \text{otherwise}.
        \end{cases}
    \end{equation*}
    Since $X_i$ is a bounded above complex of injective $\opconstr$-modules, we have that it lies in $\lsubinj{\opconstr}$ by \cite[Proposition~2.2]{Rickard2019}. Moreover, $\cohomology{n}{X_i}$ is a quotient of $X_i^n$ for each $n \in \integers$, so the canonical projection morphism
    \begin{equation*}
        {X_i} \rightarrow {\ds{n}{\integers}{\cohomology{n}{X_i}[-n]}}
    \end{equation*}
    is a quasi-isomorphism. Localising subcategories are closed under quasi-isomorphisms and direct summands, so both $\simp{\iidemp{i}}$ and $\simp{\viidemp{i}}$ lie in $\lsubinj{\opconstr}$. Consequently, all the simple $\opconstr$-modules lie in $\lsubinj{\opconstr}$ by Lemma~\ref{lemma:simples_in_top_of_dual}, and injectives generate for $\opconstr$.
\end{proof}

\begin{remark} \label{remark:formal_complex_for_quiver_algebra}
    If $A$ is a quiver algebra, then the complex $X_i$ used in the proof of Proposition~\ref{prop:IG(tildeA^op)} can be illustrated via diagrams of modules. In particular, it is the complex
    \begin{center}
        \begin{tikzpicture}
        \node (1) at (3,0) {$X_i = \cdots$};
        \node (3) at (6,0) {\begin{tikzpicture}
                \node (31) at (7,1) {$\textcolor{Emerald}{\tilde{i}}$};
            \node (32) at (7.5,2) {$\textcolor{Emerald}{\tilde{i}}$};

            \draw[<-,Emerald] (31) to node[midway,right] {$\textcolor{Emerald}{\loopedge{i}}$} (32);

            \node (33) at (6.5,2) {$i$};

            \draw[<-,Emerald] (31) to node[midway,left] {$\textcolor{Emerald}{\connedge{i}}$} (33);
        \end{tikzpicture}};
        \node (4) at (9,0) {\begin{tikzpicture}
                \node (31) at (7,1) {$\textcolor{Emerald}{\tilde{i}}$};
            \node (32) at (7.5,2) {$\textcolor{Emerald}{\tilde{i}}$};

            \draw[<-,Emerald] (31) to node[midway,right] {$\textcolor{Emerald}{\loopedge{i}}$} (32);

            \node (33) at (6.5,2) {$i$};

            \draw[<-,Emerald] (31) to node[midway,left] {$\textcolor{Emerald}{\connedge{i}}$} (33);
        \end{tikzpicture}};
        \node (5) at (12,0) {\begin{tikzpicture}
                \node (31) at (7,1) {$\textcolor{Emerald}{\tilde{i}}$};
            \node (32) at (7.5,2) {$\textcolor{Emerald}{\tilde{i}}$};

            \draw[<-,Emerald] (31) to node[midway,right] {$\textcolor{Emerald}{\loopedge{i}}$} (32);

            \node (33) at (6.5,2) {$i$};

            \draw[<-,Emerald] (31) to node[midway,left] {$\textcolor{Emerald}{\connedge{i}}$} (33);
        \end{tikzpicture}};
        \node (6) at (14.5,0) {$0$};
        \node (7) at (16.5,0) {$\cdots$.};

        \draw[->] (1) to (3);
        \draw[->] (3) to node[midway,above] {$\loopedge{i}$} (4);
        \draw[->] (4) to node[midway,above] {$\loopedge{i}$} (5);
        \draw[->] (5) to (6);
        \draw[->] (6) to (7);
    \end{tikzpicture}
    \end{center}
\end{remark}

\begin{theorem} \label{thm:IG}
    Injectives generate for all finite dimensional algebras if and only if, for all finite dimensional algebras $\Lambda$, injectives generate for $\Lambda$ implies that injectives generate for $\op{\Lambda}$.
\end{theorem}

\begin{proof}
    Similarly to the proof of Theorem~\ref{thm:little_findim}, it suffices to consider basic finite dimensional algebras over algebraically closed fields. In particular, let $C$ be a finite dimensional algebra over a field $k$. Then $\bar{C} = \tens{C}{k}{\bar{k}}$ is a finite dimensional algebra over the algebraically closed field $\bar{k}$. Moreover, if injectives generate for $\bar{C}$, then injectives generate for $C$ by the application of \cite[Lemma~5.2]{Cummings2021} to the ring homomorphism $C \hookrightarrow \bar{C}$. Every finite dimensional algebra is Morita equivalent to a basic finite dimensional algebra and injective generation is invariant under Morita equivalence because it is invariant under derived equivalence \cite[Theorem~3.4]{Rickard2019}. So, there exists a basic finite dimensional algebra $A$ over $\bar{k}$ such that if injectives generate for $A$, then injectives generate for $C$.
    
    Thus, the statement follows similarly to the proof of Theorem~\ref{thm:little_findim}, using Propositions~\ref{prop:IG(tildeA^op)} and \ref{prop:IG(tildeA)_implies_IG(A)}.
\end{proof}

\begin{remark} \label{rem:PC}
    There is a dual concept to injective generation, namely projective cogeneration, that concerns the smallest triangulated subcategory of the derived category that contains the \textit{projective} modules and is closed under set indexed \textit{products}. See \cite[Section~5]{Rickard2019}. The analogous statement to Theorem~\ref{thm:IG} for projective cogeneration also holds. The proof of the analogous theorem has the same structure except we replace Propositions~\ref{prop:IG(tildeA)_implies_IG(A)} and \ref{prop:IG(tildeA^op)} by their dual statements.
\end{remark}

%% file: bibliography.tex
\bibliographystyle{alpha}